\documentclass[12pt]{amsart}
\usepackage{eurosym}
\usepackage{amsmath,amssymb,amsfonts,amscd,xcolor,hyperref,graphicx,ulem}
\usepackage[left=3.3cm,top=3.3cm,right=3.3cm,bottom=3.3cm]{geometry}
\usepackage[initials]{amsrefs}

\usepackage{comment}
\usepackage[utf8]{inputenc}
\usepackage{amstext}
\usepackage{amsfonts}
\usepackage{amsthm}
\usepackage{textcomp}
\usepackage{amssymb}
\usepackage{amscd}
\usepackage{amsmath}
\usepackage{xcolor}
\usepackage{tikz-cd}
\usepackage{caption}
\usepackage{enumitem}
\usepackage{theoremref}
\usepackage{mathtools}
\usepackage{float}
\newtheorem{theorem}{Theorem}[section]
\newtheorem{definition}[theorem]{Definition}
\newtheorem{proposition}[theorem]{Proposition}
\newtheorem{corollary}[theorem]{Corollary}
\newtheorem{lemma}[theorem]{Lemma}
\newtheorem{problem}[theorem]{Question}

\theoremstyle{definition}

\numberwithin{equation}{section}

\newcommand{\N}{\mathbb{N}}

\newcommand{\Z}{\mathbb{Z}}

\newcommand{\C}{\mathbb{C}}

\renewcommand{\epsilon}{\varepsilon}
\renewcommand{\phi}{\varphi}
\renewcommand{\emptyset}{\varnothing}

\DeclareMathOperator*\lowlim{\underline{lim}}

\begin{document}
\title[Recurrence and hypercyclicity in dissipative contexts]{The interplay between recurrence and hypercyclicity in dissipative contexts}
\author[D'Aniello]{E. D'Aniello}
\address[E. D'Aniello]{\mbox{}\newline \indent  Dipartimento di Matematica e Fisica, \newline\indent Universit\`a degli Studi della Campania ``Luigi Vanvitelli'', \newline\indent  Viale Lincoln n. 5,\newline\indent  81100 Caserta, Italy.}
\email{emma.daniello@unicampania.it}

\author[Maiuriello]{M. Maiuriello}
\address[M. Maiuriello]{\mbox{}\newline \indent  Dipartimento di Matematica e Fisica, \newline\indent Universit\`a degli Studi della Campania ``Luigi Vanvitelli'', \newline\indent  Viale Lincoln n. 5,\newline\indent  81100 Caserta, Italy.}
\email{martina.maiuriello@unicampania.it}

\author[Seoane]{J.B. Seoane--Sep\'ulveda}
\address[J.B. Seoane--Sep\'ulveda]{\mbox{}\newline\indent Instituto de Matem\'atica Interdisciplinar (IMI), \newline\indent Departamento de An\'{a}lisis Matem\'{a}tico y Matem\'atica Aplicada,\newline\indent Facultad de Ciencias Matem\'aticas, \newline\indent Plaza de Ciencias 3, \newline\indent Universidad Complutense de Madrid,\newline\indent 28040 Madrid, Spain.}
\email{jseoane@ucm.es}

\begin{abstract} 
Motivated by recent investigations \cite{Costakis, Bonilla} on the notion of recurrence in linear dynamics, we deepen into the notions of recurrence and frequent recurrence in the setting of dissipative composition operators with bounded distortion, a class of linear operators which includes backward shifts. 
Among other results, we show that these two notions are, actually, equivalent to those of hypercyclicity and frequent hypercyclicity, respectively. More  particularly, we improve \cite[Theorem 5.2]{Bonilla} and, also, provide an answer to an open question posed in \cite[Question 5.3]{Bonilla}.
\end{abstract}

\keywords{Hypercyclic operators, Frequently hypercyclic operators, Recurrence, Composition Operators, Weighted Shifts.}
\subjclass[2020]{Primary:  47A16, 47B33, 37B20;  Secondary: 37C50; 37B05}
\thanks{}
\maketitle

%##############################
\section{Introduction and preliminaries}
%##############################

Chao\-tic processes have, in the past, been usually associated to nonlinear settings. It was in 1929 that Birkhoff \cite{birkhoff1929} proved that there exists an entire function $f: \C \to \C$ whose sequence of translates
$$\{f(\, \cdot \,+ an): \, n \ge 1\}$$
(with $a \in \C \setminus \{0\}$)
approximates uniformly in compacta any entire function. This entails a rather wild dynamics for such a function $f$
under the action of a continuous {\it linear} self-mapping of \,$H(\C )$, namely, the translation operator $\tau_a g := g ( \, \cdot + a)$. Some years later, in 1952, MacLane  \cite{maclane1952} showed that the very same property also holds for some entire function $f$  under the action of the derivative operator $Dg := g'$. After these previous examples came up, and others that do not necessarily come from iterates of one
self-mapping, many analysts have invested much time and effort in studying these kinds of phenomena during the last three decades. These kind of processes started being studied under the notion of ``universality'':

\begin{definition}
Let $X$ and $Y$ be two topological spaces and $T_n:X \to Y$ $(n \in \N := \{1,2,\ldots\})$ be a sequence of continuous mappings. Then $(T_n)$ is said to be universal provided that there exists an element $x_0 \in X$, called universal for $(T_n)$, such that the orbit $\{T_nx_0: \, n \in \N\}$ of $x_0$ under $(T_n)$ is dense in $Y$.
\end{definition}

If $X$ and $Y$ are topological vector spaces and $(T_n) \subset L(X,Y)$, then the word ``universal'' is usually replaced with the term``hypercyclic'', terminology coined by Beauzamy  in \cite{beauzamy1986} and which is mainly used to designate an operator $T \in L(X)$ such that the sequence $(T^n)$ of its iterates is universal. 

Besides the notion of universality there have been a large amount of literature and evolving notions within this area, thus, let us provide a brief account of the terminology and literature that will be used in this paper. From now on, $T : X \rightarrow X$ stands for a bounded linear operator acting on a separable Banach space $X$. 

%%%%%%%%%%%%%%%%%%%%%%%
\subsection{Hypercyclicity and Recurrence. The basics}

Let us begin with some definitions and a brief background on hypercyclicity and recurrence.

\begin{definition} A vector $x \in X$ is said to be
	\begin{itemize}
		\item[(i.)] \textit{recurrent} for $T$ if there exists a strictly increasing sequence $\{n_k\}_{k \in \mathbb N}$ of positive integers such that \[T^{n_k} x \longrightarrow x \text{ as } k \longrightarrow \infty \]
		\item[(ii.)] \textit{frequently recurrent} for $T$ if, for any neighborhood $U$ of $x$, the following lower density  \[\lowlim_{N \rightarrow \infty} \frac{1}{N}\, \#\left\{1\leq n \leq N : T^{n}(x) \in U\right\}\]
		is positive
		\item[(iii.)] \textit{reiteratively recurrent} for $T$ if, for any neighborhood $U$ of $x$, the \textit{upper Banach density}  
		\[\lim_{N \rightarrow \infty} \sup_{m \geq 0} \frac{1}{N+1} \# \{m \leq n \leq m + N:T^{n}(x) \in U\}\]
		is positive.
	\end{itemize}
\end{definition}

As it has become usual in the literature, we shall denote by $R(T)$ and $FR(T)$ the set of recurrent and frequently recurrent vectors of $T$, respectively. By $RR(T)$,  we denote  the set of reiteratively recurrent vectors of $T$. 
We have the following inclusions
$$FR(T) \subseteq RR(T) \subseteq R(T).$$

\begin{definition} 
	The operator $T:X \rightarrow X$ is said to be
	\begin{itemize}
		\item[(i.)] \textit{recurrent} if the set $R(T)$ is dense in $X$;
		\item[(ii.)] \textit{frequently recurrent} if the set $FR(T)$ is dense in $X$;
		\item[(iii.)] \textit{reiteratively recurrent} if the set $RR(T)$ is dense in $X$.
	\end{itemize}
\end{definition}

\begin{definition} A vector $x \in X$ is said to be
	\begin{itemize}
		\item[(i.)]  \textit{hypercyclic} for $T$ if there exists $x$ in $X$ such that $\{T^n(x): n \in \N\}$ is dense in $X$;
		\item[(ii.)]  \textit{frequently hypercyclic} for $T$ if for each non-empty open subset $U$ of $X$ the following lower density 
		\[\lowlim_{N \rightarrow \infty} \frac{1}{N} \# \{1\leq n \leq N:T^{n}(x) \in U\}\]
		is positive; 
		\item[(iii.)]  \textit{reiteratively hypercyclic} for $T$ if for each non-empty open subset $U$ of $X$ the \textit{upper Banach density}  
		\[\lim_{N \rightarrow \infty} \sup_{m \geq 0} \frac{1}{N+1} \# \{m \leq n \leq m + N:T^{n}(x) \in U\}\]
		is positive.
	\end{itemize}
\end{definition}

\begin{definition} The operator $T:X \rightarrow X$ is said to be
	\begin{itemize}
		\item[(i.)]  \textit{hypercyclic} if $T$  admits a hypercyclic vector;
		\item[(ii.)]  \textit{frequently hypercyclic} if $T$ admits a  frequently hypercyclic vector;
		\item[(iii.)]  \textit{reiteratively hypercyclic} if $T$ admits a reiteratively hypercyclic vector.
	\end{itemize}
\end{definition}

As mentioned, hypercyclicity is a central notion in Linear Dynamics and the ``vastness'' of (frequently) hypercyclic phenomena has been recently investigated in, for instance, \cite{DAnielloMaiuriellospace, Maiuriello3,bernalcalderonetal,bernal2018}.  It is well known that, 
on separable Banach spaces, hypercyclicity is equivalent to the property that for any pair of non-empty open subsets $U$, $V$ of $X$, there is $k \in {\mathbb N}$ such that $T^{k}(U) \cap V \not= \emptyset$.  It is also known (see \cite{Costakis}) that $T$ is recurrent if for any non-empty open subset $U$ of $X$, there is $k \in {\mathbb N}$ such that $U \cap T^{-k}(U) \not= \emptyset$. Clearly, the various types of hypercyclicity  are stronger than the corresponding ones of recurrence. In \cite{GE, Maiuriello, Bonilla} the reader may find fuller and detailed discussions concerning the various notions of chaos, including different types of hypercyclicity, and recurrence. Our main aim in this paper is to 
deepen into the notions of recurrence and frequent recurrence in the framework of dissipative composition operators with bounded distortion, a class of linear operators which includes backward shifts. We shall prove, among other results, that these two notions are equivalent to,  respectively, those of hypercyclicity and frequent hypercyclicity. 

%%%%%%%%%%%%%%%%%%%%%
\subsection{Composition operators and weighted shifts}

\begin{definition} \label{compodyn}
	Let $(X,{\mathcal B},\mu)$ be a $\sigma$-finite measure space and $f : X \to X$ be a bijective bimeasurable transformation satisfying, together with $f^{-1}$,
	\begin{equation}\label{condition}
		\exists c>0 \ \ : \ \  \mu(f^{-1}(B)) \leq c \mu(B) \ \textrm{ for every } B \in {\mathcal B}.
		\tag{$\star$}
	\end{equation}
	For $1 \le p <\infty$, the \textit{composition operator induced by}\index{composition!operator} $f$ is the bounded linear operator 
	\begin{align*}
		T_f \colon L^p(X,{\mathcal B},\mu) &\rightarrow L^p(X,{\mathcal B},\mu) \\
		\varphi &\mapsto \varphi \circ f 
	\end{align*}
\end{definition}

We recall that the fact that $f$ is bijective and $f^{-1}$ satisfies ($\star$) imply that $T_f$ is invertible with $T_f^{-1}=T_{f^{-1}}$. The spaces $L^p(X,{\mathcal B},\mu)$ will 
be simply denoted by $L^p(X)$, $1 \leq p < \infty$. 

\begin{definition}
	Given a composition operator $T_f$, if there exists $W \in \mathcal B$, $0<\mu(W)<\infty$, such that 
	\[X = \displaystyle{ {\dot \bigcup_{k \in \Z}} f^k (W)},\]
	then $T_f$ is said to be a \textit{dissipative} composition operator.\\
	If, in addition, there exists $K>0$ such that
	\begin{equation*}
		\dfrac{1}{K} \mu(f^k(W))\mu(B) \leq \mu(f^k (B))\mu (W) \leq K \mu(f^k(W))\mu(B)
	\end{equation*}
	for every $k \in \mathbb Z$ and  $B \in {\mathcal B}(W)$, where  ${\mathcal B}(W) =\{ B \cap W, B \in {\mathcal B} \},$ then $T_f$ is said to be a dissipative composition operator of \textit{bounded distortion}.
\end{definition}

The work \cite{SinghManhas1993} contains fundamental results for the reader interested in composition operators in general contexts. In recent years, the dynamics of composition operators  have been widely investigated and understood (see, for instance, \cite{BADP, BDP, DAnielloDarjiMaiuriello, DAnielloDarjiMaiuriello2, Darjipires, Maiuriello2} and references therein), and the related 
results have been obtained in general contexts or, sometimes, in the particular case of dissipative contexts. For this case, see \cite{DAnielloDarjiMaiuriello}, where dissipative composition operators 
of bounded distortion appear for the first time.
These operators, as known from \cite{DAnielloDarjiMaiuriello}, are closely connected with weighted backward shifts. We recall that, given $\{w_i\}_{i \in \Z}$ a bounded sequence of positive numbers, 
a \textit{weighted backward shift with weights $\{w_i\}_{i \in \Z}$} is a bounded linear operator $B_w: \ell^p(\Z) \rightarrow \ell^p(\Z)$ defined by
$B_w({\bold x})(i) = w_{i+1}x_{i+1}$. Moreover, if $\{w_i\}_{i \in \Z}$ is bounded away from zero, then $B_w$ is invertible. For an exhaustive description of the dynamics of $B_w$ and their spectrum, 
see \cite{DAnielloMaiurielloSpectrum, Maiuriello}.\\

We conclude this section pointing out the close relation between $B_w$ and $T_f$. To this aim, we recall that two linear operators $T:X \rightarrow X$ and $S:Y \rightarrow Y$ are said to 
be \textit{topologically semi-conjugate} if there exists a linear, bounded, surjective map $\Pi: X \rightarrow Y$, called \textit{factor map,} for which $\Pi \circ T=S\circ \Pi$, i.e. the following diagram commutes:
{\large
\begin{figure}[H] 
	\centering
	\begin{tikzcd}
		X  \arrow[r, "T"] \arrow[d, "\Pi"]
		& X \arrow[d, "\Pi"] \\ Y  \arrow[r,  "S"]
		&Y
	\end{tikzcd} 
\end{figure} 
}
In such case, $S$ is called a \textit{factor} of $T$. In particular, if $\Pi$ is a homeomorphism, $S$ and $T$ are said to be \textit{topologically conjugate}. Throughout the paper, as usual, 
the symbol $\dot {\cup}$ denotes the disjoint union.

\begin{lemma}{\cite[Lemma 4.2.3]{DAnielloDarjiMaiuriello}}\label{factorBw}
	Let $T_f$ be a dissipative composition operator of bounded distortion, generated by a wandering set $W$. Consider the bilateral weighted backward shift $B_w$ with weights
	\[w_{k} =  \left( \frac{\mu(f^{k-1}(W))}{\mu(f^{k}(W))}\right)^{\frac{1}{p}}, \quad k \in \Z. \] 
	Then, $B_w$ is a factor of $T_f$ by the factor map $\Pi: L^p(X) \rightarrow \ell ^p({\mathbb Z}) $ defined as 
	\[\Pi(\varphi)=\left  \{ \dfrac{\mu(f^{k}(W))^{\frac{1}{p}}}{\mu(W)} \int_{W}  \varphi \circ f^{k} d \mu \right \}_{k \in \Z}  \]
\end{lemma}

A detailed description of the above relation, as well as the shift-like behavior of such kind of composition operators, is provided in \cite{DAnielloDarjiMaiuriello2, Maiuriello} (and references therein). 
We point out that composition operators include weighted shifts but the class is much larger than that: it includes, for example, operators induced by measures on odometers \cite{Bongiorno}.

%%%%%%%%%%%%%%%%%%%%%%%%%%%%%
%%%%%%%%%%%%%%%%%%%%%%%%%%%%%
%%%%%%%%%%%%%%%%%%%%%%%%%%%%%
%%%%%%%%%%%%%%%%%%%%%%%%%%%%%
\section{Recent advances on recurrence}
%%%%%%%%%%%%%%%%%%%%%%%%%%%%%
%%%%%%%%%%%%%%%%%%%%%%%%%%%%%
%%%%%%%%%%%%%%%%%%%%%%%%%%%%%
%%%%%%%%%%%%%%%%%%%%%%%%%%%%%

As already recalled, hypercyclicity implies recurrence. In \cite{Bonilla}, the authors prove that the two properties, of being recurrent and being hypercyclic,  are equivalent in the context of weighted shifts. 
Specifically, they prove the following theorem (which was first proved in \cite[Proposition 5.1]{Costakis} for the Hilbert space $\ell^2(\mathbb Z)$).

\begin{theorem}\label{Bonilla-et-All}\cite[Theorem 5.2]{Bonilla}
	Let $B_w$ be a bilateral weighted shifts, with weights $w=\{w_n\}_{n \in \mathbb Z}$. If $B_w$ is recurrent, then it is hypercyclic.
\end{theorem}

In the following result,  Theorem \ref{thmrecurrent}, we prove that the two properties, of being recurrent and of being hypercyclic, are equivalent for invertible dissipative composition operators $T_f$. We extend, in this way, 
Theorem \ref{Bonilla-et-All}, as each bilateral weighted shifts is (topologically conjugated to) an invertible dissipative composition operator $T_f$ (\cite{DAnielloDarjiMaiuriello2, Maiuriello}).\\
We recall that, if $T_f$ is invertible, then $T_f$ is recurrent if and only if $T_f^{-1}$ is recurrent \cite[Remark 1.5]{Costakis}.

\begin{theorem} \label{thmrecurrent}
	Let $T_f$ be a dissipative composition operator. Then, $T_f$ is hypercyclic if and only if $T_f$ is recurrent.
\end{theorem}
\begin{proof}
	If $T_f$ is hypercyclic then, clearly, $T_f$ is recurrent. We only need to prove the converse. In order to achieve this, let us assume that $T_f$ is recurrent. Let $N$ be a positive integer and consider the open ball 
	\[ B_{\delta}(\chi_{\dot{\cup}_{\vert j \vert \leq N} f^j(W)})=\{ \varphi \in L^p(X) : \Vert \chi_{\dot{\cup}_{\vert j \vert \leq N} f^j(W)} - \varphi \Vert_p < \delta \}.\] 
	There exists a positive integer $n >2N$ such that 
	 
	 \begin{equation}\label{heart}
 B_{\delta}(\chi_{\dot{\cup}_{\vert j \vert \leq N} f^j(W)}) \cap T_f^{-n}( B_{\delta}(\chi_{\dot{\cup}_{\vert j \vert \leq N} f^j(W)})) \neq \emptyset. 
	\end{equation}
	
	Moreover, as also $T_f^{-1}$ is recurrent, there exists a positive integer $m >2N$ such that  
	
		 \begin{equation}\label{hearts}
		 	 B_{\delta}(\chi_{\dot{\cup}_{\vert j \vert \leq N} f^j(W)}) \cap T_f^{m}( B_{\delta}(\chi_{\dot{\cup}_{\vert j \vert \leq N} f^j(W)})) \neq \emptyset  .
		 	\end{equation}

	Let $\epsilon >0$. Let $c$ and $d$ be the constants of condition $(\star)$ for $f$ and $f^{-1}$, respectively. As $c,d,n,m$ are fixed numbers, consider $\delta >0$ such that 
	\[\delta < \min \left \{\dfrac{\epsilon^{\frac{1}{p}}}{2}, \dfrac{\epsilon^{\frac{1}{p}}}{2 (\max\{c, d\})^{\frac{\vert m-n\vert}{p}}} \right \}.\]
	By considering \eqref{heart}, we have 
	\[ B_{\delta}(\chi_{\dot{\cup}_{\vert j \vert \leq N} f^j(W)}) \cap B_{\delta}(\chi_{\dot{\cup}_{\vert j \vert \leq N} f^j(W)} \circ f^{-n}) \neq \emptyset. \]
	Therefore, there exists $\varphi \in B_{\delta}(\chi_{\dot{\cup}_{\vert j \vert \leq N} f^j(W)}) \cap B_{\delta}(\chi_{\dot{\cup}_{\vert j \vert \leq N} f^j(W)} \circ f^{-n})$, meaning that 
	
	\[\Vert \varphi - \chi_{\dot{\cup}_{\vert j \vert \leq N} f^j(W)} \Vert_p < \delta \ \ \  \& \ \ \ \Vert \varphi - (\chi_{\dot{\cup}_{\vert j \vert \leq N} f^j(W)} \circ f^{-n}) \Vert_p < \delta.\]
	The above inequalities, together with the fact that $\chi_{\dot{\cup}_{\vert j \vert \leq N} f^j(W)}=0$ on $\dot{\cup}_{\vert j \vert \leq N} f^{n+j}(W)$ (as $n > 2N$), imply the following
	\begin{eqnarray*}
		(\mu(f^n(\dot{\cup}_{\vert j \vert \leq N} f^j(W))))^{\frac{1}{p}} &=& (\mu(\dot{\cup}_{\vert j \vert \leq N} f^{n+j}(W)))^{\frac{1}{p}} \\
		&=& \left ( \int_{\dot{\cup}_{\vert j \vert \leq N} f^{n+j}(W)} \vert \chi_{\dot{\cup}_{\vert j \vert \leq N} f^j(W)} \circ f^{-n} \vert^p  d\mu \right )^{\frac{1}{p}}\\
		&=& \left ( \int_{\dot{\cup}_{\vert j \vert \leq N} f^{n+j}(W)} \vert \chi_{\dot{\cup}_{\vert j \vert \leq N} f^j(W)} \circ f^{-n} - \chi_{\dot{\cup}_{\vert j \vert \leq N} f^j(W)}d\mu  \vert^p \right )^{\frac{1}{p}}\\
		&\leq & \left ( \int_{X} \vert \chi_{\dot{\cup}_{\vert j \vert \leq N} f^j(W)} \circ f^{-n} - \chi_{\dot{\cup}_{\vert j \vert \leq N} f^j(W)}  \vert^p d\mu \right )^{\frac{1}{p}}\\
		&=& \Vert (\chi_{\dot{\cup}_{\vert j \vert \leq N} f^j(W)} \circ f^{-n}) -  \chi_{\dot{\cup}_{\vert j \vert \leq N} f^j(W)} \Vert_p \\
		&\leq& \Vert  (\chi_{\dot{\cup}_{\vert j \vert \leq N} f^j(W)} \circ f^{-n})- \varphi \Vert_p + \Vert \varphi - \chi_{\dot{\cup}_{\vert j \vert \leq N} f^j(W)}   \Vert_p \\
		&\leq& 2\delta.
	\end{eqnarray*}
	Hence, 
	
	\begin{equation}\label{bullet1}
	\mu(f^n(\dot{\cup}_{\vert j \vert \leq N} f^j(W))) \leq 2^p\delta^p.
	\end{equation}

	By considering now \eqref{hearts} we obtain, by proceeding as above, 
	\[\mu(f^{-m}(\dot{\cup}_{\vert j \vert \leq N} f^j(W))) \leq 2^p\delta^p.\]
	Now, note that
	
	\begin{itemize}
		\item[$\bullet$] if $m <n$, then
			\begin{equation}\label{bullet2}
		 \mu(f^{-n}(\dot{\cup}_{\vert j \vert \leq N} f^j(W))) \leq c^{\vert n-m \vert } \mu(f^{-m}(\dot{\cup}_{\vert j \vert \leq N} f^j(W)))\leq  c^{\vert n-m \vert } 2^p\delta^p
	\end{equation}
		
		\item[$\bullet$] if $m \geq n$, then 
			\begin{equation}\label{bullet3}
				 \mu(f^{-n}(\dot{\cup}_{\vert j \vert \leq N} f^j(W))) \leq d^{\vert n-m \vert} \mu(f^{-m}(\dot{\cup}_{\vert j \vert \leq N} f^j(W)))\leq d^{\vert n-m \vert} 2^p\delta^p.
				\end{equation}
	\end{itemize}
	
	\noindent
	Hence, by the arbitrariness of $\epsilon$ and $N$, and by the choice of $\delta$, from \eqref{bullet1}, \eqref{bullet2} and \eqref{bullet3} it follows that 
	\[\mu(f^{n}(\dot{\cup}_{\vert j \vert \leq N} f^j(W))) < \epsilon \ \ \ \& \ \ \ \mu(f^{-n}(\dot{\cup}_{\vert j \vert \leq N} f^j(W))) < \epsilon.\]
	The above conditions, as proved in \cite[Proposition 5.1]{DAnielloDarjiMaiuriello2}, are sufficient to guarantee that $T_f$ is hypercyclic. 
	
\end{proof}

\begin{corollary}
	Let $T_f$ be a dissipative composition operator of bounded distortion. Let $B_w$ be the associated weighted backward shift, with weights 
	$$w_{k}=  \left( \frac{\mu(f^{k-1}(W))}{\mu(f^{k}(W))}\right)^{\frac{1}{p}}, \quad k \in \mathbb Z.$$ Then, $T_f$ is recurrent if and only if $B_w$ is recurrent.
\end{corollary}
\begin{proof}
	From Theorem \ref{thmrecurrent}, it follows that $T_f$ is recurrent if and only if it is hypercyclic. As mentioned above, the two notions are equivalent also 
	for weighted backward shifts. Moreover, by \cite[Theorem M]{DAnielloDarjiMaiuriello2}, $T_f$ is hypercyclic if and only if $B_w$ is so, from which the thesis follows.
\end{proof}

\begin{problem}\label{Q1}
In \cite[Question 5.3]{Bonilla}, the authors ask the following:  Does the analogue of Theorem \ref{Bonilla-et-All} hold for frequent recurrence? That is, if $B_w$ is frequently recurrent then, is it true that it is frequently hypercyclic?
\end{problem}

We can, actually, extend this previous question to the following one: 

\begin{problem}\label{Q2}
 What about the analogous for $T_f$ in a dissipative setting?
\end{problem}

Next, in Proposition \ref{prop1}, we provide a positive answer to Question \ref{Q1}. We first need to recall the following result. 

\begin{theorem}{\cite[Theorem 5.5]{Bonilla}} \label{thmBW}
	Let $B_w$ be a weighted backward shift on $\ell^p({\mathbb N})$ or $\ell^p({\mathbb Z})$, $1 \leq p < \infty$. If $B_w$ admits a non-zero reiteratively recurrent vector, then it is chaotic and therefore frequently hypercyclic.
\end{theorem}

\begin{proposition} \label{prop1}
	Let $B_w$ be a bilateral weighted backward shift, with weights $w=\{w_{k}\}_{k \in \mathbb Z}$.   Then, $B_w$ is frequently recurrent if and only if $B_w$ is frequently hypercyclic.
\end{proposition}
\begin{proof}
	As already mentioned, frequent hypercyclicity always implies frequent recurrence. Hence, we only need to show that if $B_w$ is frequently recurrent then $B_w$ is frequently hypercyclic. In order to obtain this, let us assume that $B_w$ is frequently recurrent, 
	i.e., that $FR(B_w)$ is dense in $\ell^p(\mathbb Z)$. By contradiction, let $B_w$ not be frequently hypercyclic. Hence, by Theorem \ref{thmBW}, $B_w$ cannot have non-zero reiteratively recurrent vectors. As  $FR(B_w) \subseteq RR(B_{w})$, this contradicts the fact that $FR(B_w)$ is dense in $\ell^p(\mathbb Z)$. Therefore, $B_w$ must be frequently hypercyclic.
\end{proof}

\begin{corollary} \label{cor1}
	Let $T_f$ be a dissipative composition operator of bounded distortion. Let $B_w$ be the associated weighted backward shift, with weights $w_{k}=  \left( \frac{\mu(f^{k-1}(W))}{\mu(f^{k}(W))}\right)^{\frac{1}{p}}$, ${k \in \mathbb Z}$. 
	Then, $T_f$ is frequently recurrent if and only if $B_w$ is frequently recurrent.
\end{corollary}
\begin{proof}
	Firstly, suppose that $B_w$ is frequently recurrent. Then, by Proposition \ref{prop1}, $B_w$ is frequently hypercyclic. This, by \cite[Theorem M]{DAnielloDarjiMaiuriello2} is equivalent to have $T_f$ 
	frequently hypercyclic, and therefore $T_f$ frequently recurrent.
	
	Secondly, assume now that $T_f$ frequently recurrent. This means that $FR(T_f)$ is dense in $L^p(X)$, i.e. $\overline{FR(T_f)}=L^p(X)$. Moreover, as $T_f$ and $B_w$ are semiconjugate via $\Pi: L^p(X) \rightarrow \ell^p(\mathbb Z)$, by arguing as in  \cite[Proposition 1.19]{GE}, it follows that if $\varphi$ is a frequently recurrent vector of $T_f$, then $\Pi(\varphi)$ is a frequently recurrent vector for $B_w$. Then, $\Pi(FR(T_f)) \subseteq FR(B_w)$. We recall that $\Pi$ is surjective and, hence, $\Pi(L^p(X))= \ell^p(\mathbb Z)$. Therefore, 
	\[ \ell^p(\mathbb Z)=\Pi(L^p(X))=\Pi(\overline{FR(T_f)}) \subseteq \overline{\Pi(FR(T_f))} \subseteq \overline{FR(B_w)} \subseteq \ell^p(\mathbb Z)\]
	implying  $\ell^p(\mathbb Z)= \overline{FR(B_w)}$. That is, $B_w$ is frequently recurrent.
\end{proof}

By using the above Corollary \ref{cor1}, we obtain the following theorem, which provides an answer to Question \ref{Q2}, and extends Proposition \ref{prop1} . 

\begin{theorem}
Let $T_f$ be a dissipative composition operator of bounded distortion. Then, $T_f$ is frequently recurrent if and only if $T_f$ is frequently hypercyclic.
\end{theorem}

\begin{proof}
	By Corollary \ref{cor1}, $T_f$ is frequently recurrent if and only if $B_w$ is frequently recurrent. This is equivalent, by Proposition \ref{prop1}, to have $B_w$ frequently hypercyclic. 
	By \cite[Theorem M]{DAnielloDarjiMaiuriello2}, this is equivalent to $T_f$ frequently hypercyclic.
\end{proof}

\bibliographystyle{siam}
\bibliography{biblio}

%%%%%%%%%%
%%%%%%%%%%
%REFERENCES 
%%%%%%%%%%
%%%%%%%%%%
%%%%%%%%%%

%\begin{bibdiv}
%\begin{biblist}
%
%
%\end{biblist}
%\end{bibdiv}
\end{document}